\documentclass[12 pt,twoside]{amsart}
\usepackage{amsmath}
\usepackage{amssymb}
\usepackage{hyperref}
 \usepackage{amsthm}

\begin{document}

\newtheorem{theorem}{Theorem}[section]
\newtheorem{corollary}[theorem]{Corollary}
\newtheorem{proposition}[theorem]{Proposition}
\newtheorem{lemma}[theorem]{Lemma}
\theoremstyle{definition}
\newtheorem{definition}[theorem]{Definition}
\theoremstyle{remark}
\newtheorem{remark}[theorem]{\bf Remark}
\newcommand{\ds}{\displaystyle}
\newcommand{\R}{\mathbb{R}}
\newcommand{\noi}{\noindent}
\newcommand{\ts}{\thinspace}
\newcommand{\leqn}[1]{\\*\\* \indent #1 \\*\\*}
\newcommand{\leqnn}[1]{\\*\\* \indent #1 \\*}
\newcommand{\bleqn}[1]{\\*\\*\\* \indent #1 \\*}
\renewcommand{\thesubsection}{}
\newcommand{\norm}[1]{\left \lVert #1 \right \rVert}
\newcommand{\biarraysmall}[2]{\begin{array}{c} #1 \vspace{-.20cm}\\ #2  \end{array}}
\newcommand{\biarray}[2]{\begin{array}{c} #1 \vspace{.3cm}\\ #2  \end{array}}
\newcommand{\triarray}[3]{\begin{array}{c} #1 \vspace{.15cm}\\ #2 \vspace{.15cm}\\ #3 \end{array}}
\newcommand{\quadarray}[4]{\begin{array}{c} #1 \vspace{.15cm}\\ #2 \vspace{.15cm}\\ #3 \vspace{.15cm}\\ #4 \end{array}}
\newcommand{\quintarray}[5]{\begin{array}{c} #1 \vspace{.15cm}\\ #2 \vspace{.15cm}\\ #3 \vspace{.15cm}\\ #4 \vspace{.15cm}\\ #5 \end{array}}
\newcommand{\barr}{\begin{eqnarray}}
\newcommand{\beq}{\begin{equation}}
\newcommand{\bpf}{\begin{proof} \quad}
\newcommand{\btm}{\begin{thm}}
\newcommand{\blem}{\begin{lem}}
\newcommand{\elem}{\end{lem}}
\newcommand{\earr}{\end{eqnarray}}
\newcommand{\eeq}{\end{equation}}
\newcommand{\epf}{\end{proof}}
\newcommand{\etm}{\end{thm}}
\newcommand{\beqq}{\begin{equation*}}
\newcommand{\eeqq}{\end{equation*}}
\newcommand{\ep}{\varepsilon}

\title[Nonlinear multipoint bvp's]{Nonlinear scalar multipoint boundary value problems at resonance }
\author[Daniel Maroncelli]{Daniel Maroncelli}
\address{Department of Mathematics, College of Charleston, Charleston, SC 29424-0001, U.S.A}
\email{maroncellidm@cofc.edu}
\maketitle

\begin{center}\textbf{{\tiny($^{**}$The final version of this manuscript has been accepted for publication in Journal of Difference Equations and Applications)}}\end{center}

{\footnotesize
\noindent
{\bf Abstract.}   In this work we provide conditions for the existence of solutions to nonlinear boundary value problems of the form
\begin{equation*}
y(t+n)+a_{n-1}(t)y(t+n-1)+\cdots a_0(t)y(t)=g(t,y(t+m-1))
\end{equation*}
subject to
\begin{equation*}
\sum_{j=1}^nb_{ij}(0)y(j-1)+\sum_{j=1}^nb_{ij}(1)y(j)+\cdots+\sum_{j=1}^nb_{ij}(N)y(j+N-1)=0
\end{equation*}
for $i=1,\cdots, n$.
The existence of solutions will be proved under a mild growth condition on the nonlinearity, $g$,  which must hold only on a bounded subset of $\{0,\cdots, N\}\times\R$.




\section{Introduction}
\setlength{\parskip}{.25cm plus 0.05cm minus 0.02cm}

 In this paper we provide criteria for the solvability of nonlinear scalar multipoint  boundary value problems of the form
\begin{equation}
y(t+n)+a_{n-1}(t)y(t+n-1)+\cdots a_0(t)y(t)=g(t,y(t+m-1))\label{E:Scalar}
\end{equation}
subject to
\begin{equation}
\sum_{j=1}^nb_{ij}(0)y(j-1)+\sum_{j=1}^nb_{ij}(1)y(j)+\cdots+\sum_{j=1}^nb_{ij}(N)y(j+N-1)=0\label{E:SBoundary}
\end{equation}
for $i=1,\cdots, n$.

Throughout our discussion we will assume that $g:\R\times\mathbb{R}\to\R$ is continuous, $m$ is fixed with $1\leq m\leq n$, $N$ is an integer greater than 2, the coefficients $b_{ij}(\cdot)$ and $a_0(\cdot), \cdots, a_{n-1}(\cdot)$ are real-valued with $a_0(t)\neq 0$ for all $t$,
and the boundary conditions are independent.

We focus on the solvability of nonlinear boundary value problems at resonance; that is, problems where the solution space of the associated  linear homogeneous problem, (\ref{A:Homog}), subject to boundary conditions, (\ref{E:SBoundary1}), is nontrivial. We will assume throughout that the solution space of this  linear homogeneous problem is $1$-dimensional.

In a vast majority of the literature on resonant boundary value problems, see \cite{CesariAbstract, EthridgePeriodic,LandesmanElliptic,LazerHarmonic,MaroncelliDiscreteBounded,MaroncelliNonlinearImp,Pollack,RodSturm,Etheridgesecondorder,RodriguezProjection,RodTaylorscalardiscrete,RodTaylorode,Rod2008}, it is assumed that the nonlinearities of the difference (differential) equation are bounded. Recently, there has been a large push to obtain existence results in cases where the nonlinearity of the difference (differential) equation is unbounded. There have been several results in this regard, most of which require $g$ to satisfy a growth condition on intervals of the form $(-\infty, z_0]$ and $[z_0,\infty)$. For interested readers, we mention \cite{DrabekRes,DrabekLL,MaroncelliMulti, MaroncelliSturm,MaroncelliPeriodic,RodAbSturmGlobal,RodriguezAlternativeMethod}.
Our focus will be on the case where the nonlinearity is allowed to be unbounded, but must satisfy a mild growth condition on a bounded subset of $\{0,\cdots, N\}\times\R$.  Those readers interested in results obtained for the case of nonresonant difference equations may consult \cite{HendersonThreepoint,HendersonPost,HendersonPostMult,HendersonThompson,LiuMult,Zhou,ZhangMultipoint}.

Our main result is Theorem \ref{T:Main}, which establishes the existence of solutions to \eqref{E:Scalar}-\eqref{E:SBoundary} under suitable interaction of the solution space
of the linear homogeneous problem and the nolinearity $g$. We would like to remark that the result we obtain in Theorem \ref{T:Main} constitutes a significant generalization of the work 
found in \cite{MaSturm, RodTaylorscalardiscrete}. In \cite{MaSturm}, the author discusses the existence of solutions to \eqref{E:Scalar}-\eqref{E:SBoundary} in the special case of nonlinear Sturm-Liouville problems with standard two-point linear boundary conditions. In \cite{RodTaylorscalardiscrete}, the authors discuss the existence of solutions to \eqref{E:Scalar}-\eqref{E:SBoundary} under the assumption of bounded nonlinearities that must also satisfy a limit condition at $\pm\infty$. This limit assumption is quite standard, often referred to as a Landesman-Lazer type condition.  In section 4, we give a detailed comparison between Theorem \ref{T:Main} and the work from \cite{MaSturm, RodTaylorscalardiscrete}.

Our main tool in the analysis of Theorem \ref{T:Main} will be the application of an alternative method in combination with Brouwer's degree theory. The application of these ideas to discrete and continuous nonlinear boundary value problems is extensive.  For those readers interested in fixed point methods, coincidence degree theory, the Lyapunov-Schmidt procedure or more general alternative methods, and their application to difference and differential equations,
we suggest \cite{Chow,EthridgePeriodic,Coincidence,Iannucci, JiangSecondOrder,LazerHarmonic,Pollack,RodriguezAlternativeMethod,RodGal,Etheridgesecondorder,RodriguezProjection,RodTaylordiscrete,RodTaylorode,Rod2008,Rouche,Zhao} and the references therein.

 \section{Preliminaries}

 The nonlinear boundary value problem (\ref{E:Scalar})-(\ref{E:SBoundary}) will be viewed as an operator problem. To help facilitate in the construction of this problem,
we define

\begin{equation*}
A(t)=
\begin{pmatrix}
0&1&0&\cdots& 0\\
0&0&1&\cdots& 0\\
\vdots&\vdots&&\ddots&\vdots\\
0&0&0&\cdots&1\\
-a_0(t)&-a_1(t)&-a_2(t)&\cdots&-a_{n-1}(t)
\end{pmatrix},
\end{equation*}
 
 $f:\R\times{\R}^n\to{\R}^n $ by 
\begin{equation*}
f\begin{pmatrix}t\\x_1\\x_2\\\vdots\\x_{n-1}\\x_n\end{pmatrix}=\begin{pmatrix}0\\0\\\vdots\\0\\g(t,x_m)\end{pmatrix},
\end{equation*}

\noindent and $n\times n$ matrices $B_k, k=0,\cdots,N$, by
\begin{equation*}
(B_k)_{ij}=b_{ij}(k).
\end{equation*}

The nonlinear boundary value problem (\ref{E:Scalar})-(\ref{E:SBoundary}) is now equivalent to the nonlinear system

\begin{equation}
 x(t+1)=A(t)x(t)+f(t,x(t)) \hspace{.25cm}  \label{E:System} 
 \end{equation}                          
 subject to boundary conditions
 \begin{equation}
 \sum_{i=0}^NB_ix(i)=0. \label{E:SBoundary1} 
 \end{equation}

The underlying function spaces for our operator problem are as follows:

\begin{equation*}
 \ds{X=\left\{\phi:\left\{0,1,2,\cdots, N\right\}\to {\R}^n \left|\right.\ts \sum_{i=0}^NB_i\phi(i)=0\right\},}
\end{equation*}
\noindent and  
\begin{equation*}
\ds{Z=\left\{\phi:\left\{0,1,2,\cdots, N-1\right\}\to {\R}^n\right\}.}
\end{equation*} 

The topologies used on $X$ and $Z$ are that of the supremum norm. We use $\norm{\cdot}$ to denote both norms and we will use $|\cdot|$ to denote the standard Euclidean norm.
 
We define operators as follows:

 $\mathcal{L}: X\to Z$  by
 \begin{equation*}
(\mathcal{L}x)(t)=x(t+1)- A(t)x(t)\label{O:L},
\end{equation*}
and

$\mathcal{F}: X \to Z$ by
\begin{equation*}\mathcal{F}(x)(t)=f(t,x(t)).
\end{equation*}

\noindent Solving the nonlinear boundary value problem (\ref{E:Scalar})-(\ref{E:SBoundary}) is now equivalent to solving 
\begin{equation}
\mathcal{L}x=\mathcal{F}(x).\label{E:SOper}
\end{equation}

\begin{remark}\label{R:Bound}It will be important to know that the very natural assumption regarding the independence of the boundary conditions is equivalent to the augmented matrix $[B_0,\cdots, B_N]$ having full row rank, and thus is also equivalent to  $Ker(\cap_{k=0}^N B_k^T)=\{0\}$.  See Definition \ref{D:Proj} and \cite{RodTaylordiscrete}.
\end{remark}

 
Crucial to the use of any alternative method is the construction of projections onto the kernel and image of $\mathcal{L}$; to aid in the construction of these projections, we obtain a complete description of these spaces.  The following characterization of kernel and image of $\mathcal{L}$ can be found in \cite{MaroncelliDiscreteBounded}.
 
 Let 
\begin{equation*}\Phi(t)=\begin{cases}I&\hspace{.5cm} \text{ if }  t=0\\
A(t-1)A(t-2)\cdots A(0) &\hspace{.5cm} \text{ if }t=1,2, \cdots
\end{cases}.
\end{equation*}
We then have that $\Phi$ is the principal fundamental matrix solution to linear homogeneous problem
 \begin{equation}
  x(t+1)=A(t)x(t)\label{A:Homog}.
  \end{equation} 
For those readers interested in the general theory of difference equations, we suggest \cite{Elaydi,Kelley}.

\begin{proposition}
\label{P:Perp}An element $h\in Z$  is contained in the $Im(\mathcal{L})$ if and only if 
\begin{equation*}B_1\Phi(1)\Phi^{-1}(1)h(0)+\cdots+B_N\Phi(N)\sum_{i=0}^{N-1}\Phi^{-1}(i+1)h(i)\in Ker\left(\left(\sum_{i=0}^NB_i\Phi(i)\right)^T\right)^\perp.\end{equation*} 
 \end{proposition}

The proof of Proposition \ref{P:Perp} is trivial and can be found in \cite{MaroncelliDiscreteBounded}.  It follows easily from the variation of parameters formula,
\begin{equation}
 x(t)=\Phi(t)x(0)+\Phi(t)\sum_{i=0}^{t-1}\Phi^{-1}(i+1)h(i)\label{E:Var},
\end{equation}
 and an application of the boundary conditions (\ref{E:SBoundary1}).

\begin{proposition}
\label{P:KerL}$\ds{ Ker\left(\sum_{i=0}^NB_i\Phi(i)\right)}$ and the solution space of the linear homogeneous problem, (\ref{A:Homog}), subject to the boundary conditions ,(\ref{E:SBoundary1}),  have the same dimension.
\end{proposition}

\begin{proof}
 Taking $h=0$ in the variation of parameters formula, (\ref{E:Var}), and applying the boundary conditions, we have
 \begin{equation*} 
 \mathcal{L}x=0 
 \text{ if and only if } \exists u\in {\R}^n \text{ such that }  x(\cdot)=\Phi(\cdot)u \text{ and }\sum_{i=0}^NB_i\Phi(i)u=0.
 \end{equation*} \end{proof}

Since we are assuming  that the solution space of the linear homogeneous problem, (\ref{A:Homog}), subject to boundary conditions, (\ref{E:SBoundary1}), is 1-dimensional, it follows from Proposition \ref{P:KerL} that  we may pick a vector $u\in {\R}^n$ which forms a  basis for  $\ds{Ker\left(\sum_{i=0}^NB_i\Phi(i)\right)}$. We define $S:\{0,1,\cdots,N\}\to{\R}^n$ by  \beqq S(t)=\Phi(t)u.\eeqq
It follows that a function $x\in Ker(\mathcal{L})$  if and only if $x(\cdot)=S(\cdot)\alpha$ for some $\alpha\in {\R}$.

Using the fact that $\ds{Ker\left(\sum_{i=0}^NB_i\Phi(i)\right)}$ and $\ds{Ker\left(\left(\sum_{i=0}^NB_i\Phi(i)\right)^T\right)}$ have the same dimension, we may also pick a vector $w\in {\R}^n$  which forms a basis for $\ds{Ker\left(\left(\sum_{i=0}^NB_i\Phi(i)\right)^T\right)}.$ We introduce the following notation which simplifies our characterization of $Im(\mathcal{L})$. We define $\ds{\Psi^T:\{0,1,2,\cdots,N-1\}\to {\R}^n}$ by 
\beqq\ds{\Psi^T(t)=\sum_{i=t+1}^Nw^TB_i\Phi(i)\Phi^{-1}(t+1)}.\eeqq

 \noindent We now have the following characterization of the $Im(\mathcal{L})$.
\begin{proposition}\label{C:CharImL}An element $h\in Z$  is contained in the $Im(\mathcal{L})$ if and only if $\ds{\sum_{i=0}^{N-1}\Psi^{T}(i)h(i)=0}$.
\end{proposition}

Having characterized the kernel and image of $\mathcal{L}$, we are now  in a position to construct the projections which will form the basis of the Lyapunov-Schmidt projection scheme.  In this regard, we choose to follow \cite{MaroncelliDiscreteBounded,RodTaylordiscrete}.
The proofs that the following operators, $P$ and $I-Q$, are projections onto the kernel and image of $\mathcal{L}$, respectively, are simple consequences of our previous characterization of these spaces.  Proofs may be found in \cite{RodTaylordiscrete}.
\begin{definition} Let $V\colon{\R}^n\rightarrow{\R}^n$ be the orthogonal projection onto $\ds{Ker\left(\sum_{i=0}^NB_i\Phi(i)\right)}$. 
Define $P\colon X\rightarrow X$  by
\begin{equation*}
[Px](t)=\Phi(t)Vx(0).
\end{equation*}
Then $P$ is a projection onto $Ker(\mathcal{L})$.
\end{definition}

 \begin{definition} \label{D:Proj}Define $Q: Z\rightarrow Z$ by 
 \beqq
[Qh](t)=\Psi(t)\left(\sum_{j=0}^{N-1}|\Psi(j)|^2\right)^{-1}\sum_{i=0}^{N-1}\Psi^T(i)h(i).
\eeqq
Then $I-Q$ is a projection onto $Im(\mathcal{L})$.
 \end{definition}
\begin{remark}That $Q$ is well-defined is a consequence of Remark \ref{R:Bound}, see \cite{RodTaylordiscrete}.\end{remark}
 

 The following is the formulation of the alternative problem which we will use to analyze the nonlinear boundary value problem, \eqref{E:Scalar}, subject to boundary conditions, \eqref{E:SBoundary}. It is often referred to as the Lyapunov-Schmidt projection scheme. This type of projection scheme has become quite standard in resonant boundary value problems, we include the proof simply for the convenience of the reader.
 \begin{proposition}
Solving  $\mathcal{L}x=\mathcal{F}(x)$ is equivalent to solving the following system
\begin{equation*}
 \left\{\begin{array}{c}
 x-Px=M_p(I-Q)\mathcal{F}(x) \\\\
\text{and}\\
\ds \sum_{i=0}^{N-1}[\Psi(i)]_ng(i,[x(i)]_m)=0
 \end{array}
 \right.,
 \end{equation*}
where $M_p$ is $\ds{\left(\mathcal{L}_{|Ker(P)}\right)^{-1}}$ and $[e]_k$ denotes the $kth$ row of a vector $e$ in $\R^n$. \label{P:Lyap}
\end{proposition}

\begin{proof}
\begin{equation*}
\begin{split}
 \mathcal{L}x=\mathcal{F}(x)&\Longleftrightarrow \left\{\begin{array}{c} (I-Q)(\mathcal{L}x-\mathcal{F}(x))=0 \\ \text{ and }\\Q(\mathcal{L}x-\mathcal{F}(x))=0 \end{array} \right. \\
 & \Longleftrightarrow \left\{ \begin{array}{c} \mathcal{L}x-(I-Q)\mathcal{F}(x)=0 \\ \text{ and }\\Q\mathcal{F}(x)=0 \end{array} \right.\\
 &\Longleftrightarrow \left\{\begin{array}{c} M_p\mathcal{L}x-M_p(I-Q)\mathcal{F}(x)=0  \\ \text{ and }\\ Q\mathcal{F}(x)=0\end{array} \right.\\ 
 &\Longleftrightarrow\left\{\begin{array}{c}  {(I-P)x-M_p(I-Q)\mathcal{F}(x)=0}\\\text{ and }\\ \ds{\sum_{i=0}^{N-1}\Psi^{T}(i)f(i,x(i))=0}\end{array} \right.\\
&\Longleftrightarrow\left\{\begin{array}{c}  {(I-P)x-M_p(I-Q)\mathcal{F}(x)=0}\\\text{ and }\\ \ds{\sum_{i=0}^{N-1}[\Psi(i)]_ng(i,[x(i)]_m)=0}\end{array} \right..
 \end{split}
 \end{equation*}
\end{proof}

\section{Main Results}
We now come to our main result.  We start by introducing some notation that will be useful in what follows. We introduce the following sets:
\begin{equation*}
\mathcal{O}_{++}=\{t\in\{0,1,\cdots, N-1\}\mid [\Psi(t)]_n>0, [S(t)]_m>0\},
\end{equation*}
\begin{equation*}
\mathcal{O}_{+-}=\{t\in\{0,1,\cdots, N-1\}\mid [\Psi(t)]_n>0, [S(t)]_m<0\},
\end{equation*}
\begin{equation*}
\mathcal{O}_{-+}=\{t\in\{0,1,\cdots, N-1\}\mid [\Psi(t)]_n<0, [S(t)]_m>0\},
\end{equation*}
\begin{equation*}
\mathcal{O}_{--}=\{t\in\{0,1,\cdots, N-1\}\mid [\Psi(t)]_n>0, [S(t)]_m<0\},
\end{equation*}
\begin{equation*}
\mathcal{O}_{0}=\{t\in\{0,1,\cdots, N-1\}\mid [\Psi(t)]_n\neq 0, [S(t)]_m=0\}, 
\end{equation*}
and 
\begin{equation*}
\mathcal{O}=\mathcal{O}_{++}\cup\mathcal{O}_{+-}\cup \mathcal{O}_{-+}\cup\mathcal{O}_{--}. 
\end{equation*}
We also define  $A:=\norm{M_p(I-Q)}$ (Operator norm), $\ds s_{\max}:=\max_{t\in\{0,\cdots, N-1\}}|[S(t)]_m|$, $\ds s_{\min}:=\min_\mathcal{O}|[S(t)]_m|$, $\norm{g}_r=\sup_{t\in\{0,\cdots, N-1\}, x\in[-r,r]}|g(t,x)|$, and $p:\R\times Im(I-P)\to Im(I-P)$ by 
\beqq
p(\alpha,v)=M_p(I-Q)\mathcal{F}(\alpha S(\cdot)+v).
\eeqq
 \begin{theorem}\label{T:Main} Suppose the following conditions hold:
 
  \vspace{.25cm}
 
 \begin{enumerate}
 
 \item[C1.] $\mathcal{O}_0$ is empty
 \item[C2.] There exists positive real numbers $c$ and $d$, with $c<d$,  and functions $W_1, U_1,  W_2, U_2, w_1 , u_1, w_2$ and $u_2$ such that 
 \begin{equation*}
 \begin{split}
 &\text{if } x\in[c,d], \text{ then }W_1(t)< g(t,x) \text{ for }\ts\ts t\in \mathcal{O}_{++}\\
 &\text{if } x\in[-d,-c], \text{ then } g(t,x)< U_1(t) \text{ for } \ts\ts t\in \mathcal{O}_{++}\\
 &\text{if } x\in[c,d], \text{ then }  g(t,x)< u_1(t) \text{ for } \ts\ts t\in \mathcal{O}_{+-}\\
 &\text{if } x\in[-d,-c], \text{ then } w_1(t)< g(t,x) \text{ for } \ts\ts t\in \mathcal{O}_{+-}\\
 &\text{if } x\in[c,d], \text{ then } g(t,x)< W_2(t) \text{ for } \ts\ts t\in \mathcal{O}_{-+}\\
 &\text{if } x\in[-d,-c], \text{ then } U_2(t)< g(t,x) \text{ for } \ts\ts t\in \mathcal{O}_{-+}\\
 &\text{if } x\in[c,d], \text{ then } u_2(t)< g(t,x) \text{ for } \ts\ts t\in \mathcal{O}_{--}\\
  and\\
 &\text{if } x\in[-d,-c], \text{ then } g(t,x)< w_2(t) \text{ for } \ts\ts t\in \mathcal{O}_{--}\\
 \end{split}
 \end{equation*}
 \vspace{.25cm}
\item[C3.]$\ds d>{cs_{\max}+A\norm{g}_d(s_{\max}+s_{\min})\over s_{\min}}$
\vspace{.25cm}
\item[C4.] $J_2\leq0\leq J_1$, where 
\begin{equation*}
J_1=\sum_{i=0}^{N-1}[\Psi(i)]_nK_1(i),
\end{equation*}
\begin{equation*}
J_2=\sum_{i=0}^{N-1}[\Psi(i)]_nK_2(i),
\end{equation*}
and $K_1$ and $K_2$ are defined by
\begin{equation*}
K_1(t)=\begin{cases}
W_1(t)& t\in \mathcal{O}_{++}\\
w_1(t)& t\in \mathcal{O}_{+-}\\
W_2(t)& t\in \mathcal{O}_{-+}\\
w_2(t)& t\in \mathcal{O}_{--}\\
\end{cases},
\end{equation*}
and
\begin{equation*}
K_2(t)=\begin{cases}
U_1(t)& t\in \mathcal{O}_{++}\\
u_1(t)& t\in \mathcal{O}_{+-}\\
U_2(t)& t\in \mathcal{O}_{-+}\\
u_2(t)& t\in \mathcal{O}_{--}\\
\end{cases}.
\end{equation*}
\vspace{.25cm}
 \hspace{-1.75cm}Then there exists a solution to the nonlinear boundary value problem \eqref{E:Scalar}-\eqref{E:SBoundary}.
\end{enumerate}
 \end{theorem}

\begin{proof}

Define $H:\R\times Im(I-P)\to \R\times Im(I-P)$ by 
\begin{equation}
H(\alpha,x)=\begin{pmatrix}
\ds\sum_{i=0}^{N-1}[\Psi(i)]_ng(i,\alpha[S(i)]_m+[p(\alpha,v)(i)]_m)
\\
v-p(\alpha,v)
\end{pmatrix}.
\end{equation}
From Proposition \ref{P:Lyap}, the zeros of $H$ are precisely the solutions of (\ref{E:Scalar})-(\ref{E:SBoundary}). We will show the existence of a solution to the nonlinear boundary value problem by showing that  the Brouwer degree of $H$, $deg(H, \Omega, 0)$, is nonzero for some appropriately chosen set $\Omega$.

To this end, endow $\R\times Im(I-P)$ with the product topology and define 
\beqq
\Omega=\{(\alpha, v)\mid |\alpha|\leq \alpha^* \text{ and } \norm{v}\leq r^*\},
\eeqq
 where $\ds\alpha^*={c+A\norm{g}_d\over s_{\min}}$ and $\ds r^*=A \norm{g}_d$.  
 
Define $Q:[0,1]\times \overline{\Omega}\to  \R\times Im(I-P)$ by 
\begin{equation*}
Q(\gamma,(\alpha,v))=\begin{pmatrix}
\ds{(1-\gamma)\alpha+\gamma\sum_{i=0}^{N-1}[\Psi(i)]_ng(i,\alpha [S(i)]_m+[p(\alpha,v)(i)]_m)}
\\
v-\gamma p(\alpha,v)
\end{pmatrix}.
\end{equation*}
It is evident that $Q$ is a homotopy between the identity mapping and $H$. In what follows, we will show that  $Q(\gamma, (\alpha,v))$ is nonzero for each $\gamma\in(0,1)$ and every $(\alpha,v)$ in $\partial(\Omega)=\{(\alpha,v)\mid |\alpha|=\alpha^* \text{ and } \norm{v}\leq r^* \text{ or } |\alpha|\leq \alpha^* \text{ and } \norm{v}= r^* \}$, so that, by the invariance of the Brouwer degree under homotopy, $deg(H,\Omega,0)=deg(I,\Omega, 0)=1$.

In what follows, it will be useful to note that
\begin{equation}
\begin{split}\label{E:alp1}
\alpha^*s_{\max}+r^*&=\left({c+A\norm{g}_d\over s_{\min}}\right)s_{\max}+ A \norm{g}_d\\
&={cs_{\max}+A\norm{g}_d s_{\max}+A\norm{g}_d s_{\min}\over s_{\min}}\\
&<d.
\end{split}
\end{equation}

We now turn our attention to showing that $Q(\gamma,(\alpha,v))\neq 0$ for each $\gamma\in(0,1)$ and every $(\alpha, v)\in\partial(\Omega)$. We start by assuming $(\alpha, v)\in \partial(\Omega)$, with $|\alpha|\leq \alpha^*$ and $\norm{v}=r^*$.

Since for every $i$,  $|\alpha [S(i)]_m+[v(i)]_m|\leq \alpha^*s_{\max}+r^*$, we have, using \eqref{E:alp1},  that $\alpha [S(i)]_m+[v(i)]_m\in [-d,d]$. It follows that
\begin{equation*}
\begin{split}\label{L:1Res}
\norm{p(\alpha,v)}&=\norm{M_p(I-Q)\mathcal{F}(\alpha S(\cdot)+v)}\\
		       &\leq\norm{M_p(I-Q)}\norm{\mathcal{F}(\alpha S(\cdot)+v)}\\
		       &=A\max_{\{0,\cdots, N-1\}}|f(i,\alpha S(i)+v(i))|\\
		       &=A\max_{\{0,\cdots, N-1\}}|g(i,\alpha[S(i)]_m+[v(i)]_m)|\\
		       &\leq A \norm{g}_d\\
		       &=r^*.
\end{split}
\end{equation*}

Thus, $\norm{p(\alpha,v)}\leq r^*=\norm{v}$ and it becomes clear that $Q(\gamma,(\alpha,v))\neq0$ for every $\gamma$ in (0,1), since $v-\gamma p(\alpha,v)\neq 0$.

We finish the proof by looking at the case when $(\alpha,v)\in\partial(\Omega)$ with $|\alpha|=\alpha^*$ and $\norm{v}\leq r^*$. Combining the fact that  $\norm{p(\alpha,v)}\leq r^*$ with \eqref{E:alp1}, we conclude that for each $i$ and for every $(\alpha,v)\in \partial(\Omega)$, $|\alpha [S(i)]_m+[p(\alpha, v)(i)]_m|\leq d$.

  Further, if $|\alpha|=\alpha^*$, then for all $i \in \mathcal{O}$, we have 
\begin{equation*}
\begin{split}
|\alpha [S(i)]_m+[p(\alpha,v)(i)]_m|&\geq \alpha^*s_{\min}-\norm{p(\alpha,v)}\\
&\geq \alpha^*s_{\min}-A\norm{g}_d\\
&=\left({c+A\norm{g}_d\over s_{\min}}\right)s_{\min}-A \norm{g}_d\\
&=c.
\end{split}
\end{equation*}
Thus, we have shown that when $(\alpha,v)\in\partial(\Omega)$ with $|\alpha|=\alpha^*$ and $\norm{v}\leq r^*$, then for all $i \in \mathcal{O}$, $|\alpha [S(i)]_m+[p(\alpha,v)(i)]_m|\in[c,d]$. 

In fact, we have shown that if $\alpha=\alpha^*$ and $i\in \mathcal{O}_{++}\cup \mathcal{O}_{-+}$, then 
$\alpha [S(i)]_m+[p(\alpha,v)(i)]_m\in[c,d]$ and if $i\in \mathcal{O}_{+-}\cup \mathcal{O}_{--}$, then $\alpha [S(i)]_m+[p(\alpha,v)(i)]_m\in[-d,-c]$. Similarly, if $\alpha=-\alpha^*$ and $i\in\mathcal{O}_{++}\cup \mathcal{O}_{-+}$, then 
$\alpha [S(i)]_m+[p(\alpha,v)(i)]_m\in[-d,-c]$ and if $i\in \mathcal{O}_{+-}\cup \mathcal{O}_{--}$, then $\alpha [S(i)]_m+[p(\alpha,v)(i)]_m\in[c,d]$. 

Using C2., we now conclude that when $\alpha=\alpha^*$,  \begin{equation*}
\begin{split}
&W_1(i)< g(i,\alpha [S(i)]_m+[p(\alpha,v)(i)]_m) \text{ for } \ts\ts i\in \mathcal{O}_{++}\\
&w_1(i)< g(i,\alpha [S(i)]_m+[p(\alpha,v)(i)]_m) \text{ for } \ts\ts i\in \mathcal{O}_{+-}\\
&g(i,\alpha [S(i)]_m+[p(\alpha,v)(i)]_m)<W_2(i) \text{ for } \ts\ts i\in \mathcal{O}_{-+}\\
\end{split}
\end{equation*}
and
\begin{equation*}
g(i,\alpha [S(i)]_m+[p(\alpha,v)(i)]_m)< w_2(i) \text{ for } \ts\ts i\in \mathcal{O}_{--}.
\end{equation*}

Thus, since $\mathcal{O}_0$ is empty,  
\beqq
\begin{split}
\sum_{i=0}^{N-1}[\Psi(i)]_ng(i,\alpha [S(i)]_m+[p(\alpha,v)(i)]_m)&=\sum_{i\in\mathcal{O}}[\Psi(i)]_ng(i,\alpha [S(i)]_m+[p(\alpha,v)(i)]_m)\\
&>\sum_\mathcal{O}[\Psi(i)]_nK_1(i)\\
&=\sum_{i=0}^{N-1}[\Psi(i)]_nK_1(i)\\
&=J_1\\
&\geq 0.
\end{split}
\eeqq

Similarly, we may conclude that if $\alpha=-\alpha^*$,
\beqq
\begin{split}
\sum_{i=0}^{N-1}[\Psi(i)]_ng(i,\alpha [S(i)]_m+[p(\alpha,v)(i)]_m)&=\sum_{i\in\mathcal{O}}[\Psi(i)]_ng(i,\alpha [S(i)]_m+[p(\alpha,v)(i)]_m)\\
&<\sum_\mathcal{O}[\Psi(i)]_nK_2(i)\\
&=\sum_{i=0}^{N-1}[\Psi(i)]_nK_2(i)\\
&=J_2\\
&\leq 0.
\end{split}
\eeqq

It follows that $\alpha\cdot\sum_{i=0}^{N-1}[\Psi(i)]_ng(i,\alpha [S(i)]_m+[p(\alpha,v)(i)]_m)>0$.  Since 
\beqq
\ds{(1-\gamma)\alpha+\gamma\sum_{i=0}^{N-1}[\Psi(i)]_ng(i,\alpha [S(i)]_m+[p(\alpha,v)(i)]_m)}
\eeqq
would be $0$ for some $\gamma\in(0,1)$ if and only if
\beqq\alpha\cdot\sum_{i=0}^{N-1}[\Psi(i)]_ng(i,\alpha [S(i)]_m+[p(\alpha,v)(i)]_m)<0,
\eeqq
 we have that $Q(\gamma,(\alpha,v))$ is nonzero in these cases. 
We now conclude, by the homotopy invariance of the Brouwer degree, that \begin{equation*}deg(H,\Omega,0)=deg(I,\Omega,0)=1.\end{equation*}
The result now follows.
\end{proof}


\begin{remark}\label{R:Reverse}
 If the inequalities of Theorem \ref{T:Main} are reversed; that is, 
 
 \begin{equation*}
 \begin{split}
 &\text{if } x\in[c,d], \text{ then }W_1(t)>g(t,x) \text{ for }\ts\ts t\in \mathcal{O}_{++}\\
 &\text{if } x\in[-d,-c], \text{ then } g(t,x)> U_1(t) \text{ for } \ts\ts t\in \mathcal{O}_{++}\\
 &\text{if } x\in[c,d], \text{ then }  g(t,x)>u_1(t) \text{ for } \ts\ts t\in \mathcal{O}_{+-}\\
 &\text{if } x\in[-d,-c], \text{ then } w_1(t)> g(t,x) \text{ for } \ts\ts t\in \mathcal{O}_{+-}\\
 &\text{if } x\in[c,d], \text{ then } g(t,x)> W_2(t) \text{ for } \ts\ts t\in \mathcal{O}_{-+}\\
 &\text{if } x\in[-d,-c], \text{ then } U_2(t)> g(t,x) \text{ for } \ts\ts t\in \mathcal{O}_{-+}\\
 &\text{if } x\in[c,d], \text{ then } u_2(t)> g(t,x) \text{ for } \ts\ts t\in \mathcal{O}_{--}\\
  and\\
 &\text{if } x\in[-d,-c], \text{ then } g(t,x)> w_2(t) \text{ for } \ts\ts t\in \mathcal{O}_{--}\\
 \end{split},
 \end{equation*}
  then provided $J_1\leq0\leq J_2$, (\ref{E:Scalar})-(\ref{E:SBoundary}) has a solution. The proof is essentially the same.
 \end{remark}
 The following corollary isolates the special case in which $[\Psi(i)]_n$ and $[S(i)]_m$ have the same sign for all $i=0,\cdots, N-1$.  This case is of special interest since it occurs in all `self-adjoint'  boundary value problems, Sturm-Liouville boundary value problems being a special case, specific cases of second-order periodic difference equations being another. It also happens in several other cases, as we will see in our example in section 5.

  
    \begin{corollary}\label{C:Sign}
 Suppose the following conditions are satisfied:
 \begin{enumerate}
 \item[C1*.] $\mathcal{O}_0$ is empty.
 \item[C2*.] $[\Psi(i)]_n\cdot [S(i)]_m\geq 0$ (or $[\Psi(i)]_n\cdot [S(i)]_m\leq 0$) for all $i=0,\cdots, N-1$.
 \item[C3*.] There exists positive real numbers $c$ and $d$, with $c<d$, such that $g(t,x)>0$ (or $g(t,x)<0$) for every $x\in [c,d]$ and each $t=0,\cdots, N-1$ and $g(t,x)<0$ (or $g(t,x)>0$) for every $x\in [-d,-c]$ and each $t=0,\cdots, N-1$.
 \item[C4*.]$\ds d>{cs_{\max}+A\norm{g}_d(s_{\max}+s_{\min})\over s_{\min}}$ 
 \end{enumerate}
 Then the nonlinear boundary value problem (\ref{E:Scalar})-(\ref{E:SBoundary}) has a solution.
 \end{corollary}
 \begin{proof}
 It suffices to show that conditions C2. and C4. of Theorem \ref{T:Main} hold.  We will assume $[\Psi(i)]_n\cdot [S(i)]_m\geq 0$, that $g(t,x)>0$  for every $x\in [c,d]$ and each $t=0,\cdots, N-1$, and that $g(t,x)<0$ for every $x\in [-d,-c]$ and each $t=0,\cdots, N-1$.  Since $[\Psi(i)]_n\cdot [S(i)]_m\geq 0$, we have that $\mathcal{O}_{+-}$ and $\mathcal{O}_{-+}$ are empty. It follows that condition C2. of Theorem \ref{T:Main} reduces to 
 
 (NC2.) There exists positive numbers $c$ and $d$, with $c<d$, and functions $W_1, U_1,  w_2$ and $u_2$ such that 
 \begin{equation*}
 \begin{split}
 &\text{if } x\in[c,d], \text{ then }W_1(t)< g(t,x) \text{ for } \ts\ts t\in \mathcal{O}_{++}\\
 &\text{if } x\in[-d,-c], \text{ then } g(t,x)<U_1(t) \text{ for } \ts\ts t\in \mathcal{O}_{++}\\
 &\text{if } x\in[c,d], \text{ then } u_2(t)<g(t,x) \text{ for } \ts\ts t\in \mathcal{O}_{--}\\
  and\\
 &\text{if } x\in[-d,-c], \text{ then } g(t,x)<w_2(t) \text{ for } \ts\ts t\in \mathcal{O}_{--}.\\
 \end{split}
 \end{equation*}
However, using C3*., NC2.  is clearly satisfied by taking $W_1=U_1=u_2=w_2=0$. It then trivially follows that $J_1=J_2=0$, so  that condition C4. of Theorem \ref{T:Main} is satisfied. This completes the proof for this case. The other cases are similar.
 \end{proof}
 
  
  The following corollary is an application of Theorem \ref{T:Main} to cases in which the nonlinearities satisfy a sublinear or `small' linear growth condition. 
 \begin{corollary}
 Suppose the following conditions are satisfied:
 
 \begin{enumerate}
  \item[C1**.] Conditions C1., C2., and C4. of Theorem \ref{T:Main} hold.
     \vspace{.25cm}
     
   \item[C2**.] There exist positive constants $M_1$ and $M_2$ such that $|g(t,x)|\leq M_1|x|^\beta+M_2$ for every $x\in [-d,d]$ and each $t=0, \cdots, N-1$, where $0<\beta\leq1$.
   \vspace{.25 cm}
  
  \item[C3**.] $\ds d>{cs_{\max}+(K_1(1-\beta)+K_2)(s_{\max}+s_{\min})\over s_{\min}-K_1\beta(s_{\max}+s_{\min})}$, where $K_1=A M_1$, $K_2=A M_2$, and we are assuming $s_{\min}-K_1\beta(s_{\max}+s_{\min})>0$; that is, $\ds K_1< {s_{\min}\over \beta(s_{\min}+s_{\max})}$.
  
  \vspace{.25cm}
 \end{enumerate}

\noindent Then the nonlinear boundary value problem, \eqref{E:Scalar}-\eqref{E:SBoundary}, has at least one solution. 
 \end{corollary}
 
 \begin{proof}
 
 From (C2**.), we get $\norm{g}_d\leq M_1d^\beta+M_2$.  Thus,
\beqq
 \begin{split}
  {cs_{\max}+A\norm{g}_d(s_{\max}+s_{\min})\over s_{\min}}&\leq {cs_{\max}+A(M_1d^\beta+M_2)(s_{\max}+s_{\min})\over s_{\min}}\\
  &={cs_{\max}+(K_1d^\beta+K_2)(s_{\max}+s_{\min})\over s_{\min}}.
  \end{split}
\eeqq
Using  (C3**.),  we have $\ds{cs_{\max}+(K_1(1-\beta)+K_2)(s_{\max}+s_{\min})\over s_{\min}-K_1\beta(s_{\max}+s_{\min})}<d$, from which we conclude 
\beqq
\begin{split}
cs_{\max}+K_2(s_{\max}+s_{\min})&<ds_{\min}-dK_1\beta(s_{\max}+s_{\min})-K_1(1-\beta)(s_{\max}+s_{\min})\\
&=ds_{\min}-K_1(1+\beta(d-1))(s_{\max}+s_{\min})\\
&\leq ds_{\min}-K_1(1+(d-1))^\beta(s_{\max}+s_{\min})\\
&= ds_{\min}-K_1d^\beta(s_{\max}+s_{\min}).
\end{split}
\eeqq
 Rearranging, it follows that 
\beqq
 \begin{split}
  {cs_{\max}+A\norm{g}_d(s_{\max}+s_{\min})\over s_{\min}}&\leq{cs_{\max}+(K_1d^\beta+K_2)(s_{\max}+s_{\min})\over s_{\min}}\\
  &<d.
  \end{split}
\eeqq
\end{proof}
\begin{remark}\label{R:LL} We would like to point out that  if $g$ is sublinear on all of $\R$; that is, there exist positive numbers $M_1, M_2$ and a constant $\beta$, $0\leq\beta< 1$, such that $|g(t,x)|\leq M_1|x|^\beta+M_2$ for every $x\in \R$ and each $t=0, \cdots, N-1$, and there is a $R>0$ such that 
\begin{equation}
 \begin{split}\label{E:LL}
 &\text{if } x>R, \text{ then }W_1(t)<g(t,x) \text{ for } \ts\ts t\in \mathcal{O}_{++}\\
 &\text{if } x<-R, \text{ then } g(t,x)< U_1(t) \text{ for } \ts\ts t\in \mathcal{O}_{++}\\
 &\text{if } x>R, \text{ then }  g(t,x)< u_1(t) \text{ for } \ts\ts t\in \mathcal{O}_{+-}\\
 &\text{if } x<-R, \text{ then } w_1(t)< g(t,x) \text{ for } \ts\ts t\in \mathcal{O}_{+-}\\
 &\text{if } x>R, \text{ then } g(t,x)< W_2(t) \text{ for } \ts\ts t\in \mathcal{O}_{-+}\\
 &\text{if } x<-R, \text{ then } U_2(t)<g(t,x) \text{ for } \ts\ts t\in \mathcal{O}_{-+}\\
 &\text{if } x>R, \text{ then } u_2(t)< g(t,x) \text{ for } \ts\ts t\in \mathcal{O}_{--}\\
  and\\
 &\text{if } x<-R, \text{ then } g(t,x)< w_2(t) \text{ for } \ts\ts t\in \mathcal{O}_{--}\\
 \end{split},
 \end{equation}
 then C3.  of Theorem \ref{T:Main} holds, since $\ds\lim_{r\to\infty}{{\ds Rs_{\max}+A\norm{g}_r(s_{\max}+s_{\min})\over\ds s_{\min}}\over r}=0<1 .$  Thus, if conditions C1. and C4. of Theorem \ref{T:Main} are satisfied, then the nonlinear boundary value problem has a solution.

In fact, if $g$ has `small' linear growth; that is, $|g(t,x)|\leq M_1|x|+M_2$ for every $x\in \R$ and each $t=0, \cdots, N-1$  with \begin{equation*}A M_1\left({s_{\max}+s_{\min}\over s_{\min}}\right)<1,\end{equation*} then provided \eqref{E:LL} holds 
  we have that C3. of Theorem \ref{T:Main} holds, since in this case  $\ds\lim_{r\to\infty}{{\ds Rs_{\max}+A\norm{g}_r(s_{\max}+s_{\min})\over\ds s_{\min}}\over r}\leq A M_1\left({s_{\max}+s_{\min}\over s_{\min}}\right)<1. $ Thus, again, if  conditions C1. and C4. of Theorem \ref{T:Main} are satisfied, then the nonlinear boundary value problem has a solution.
\end{remark}


 \section{Comparision to previous results}
In this section we show how Theorem \ref{T:Main} improves upon existing results in the literature.

\subsection{General Multipoint}
  In \cite{RodTaylorscalardiscrete} the authors look at the existence of solutions to \eqref{E:Scalar}-\eqref{E:SBoundary}.  They obtain results by placing conditions on the nonlinearity, $g$, which are much more restrictive than Theorem \ref{T:Main}.
Their main result, written in terms of the notation of this paper,  is the following:

 \begin{theorem}\label{T:App}Suppose \eqref{E:Scalar} subject to boundary conditions \eqref{E:SBoundary} has a 1-dimensional solution space.  If
 
  \vspace{.25cm}
 
 \begin{enumerate}
  \item[H1.] $g$ is independent of $t$
 \vspace{.25cm}
\item[H2.] $g(\pm\infty):=\lim_{x\to\pm\infty}g(x)$ exist

\vspace{.25cm}
\item[H3.] $\mathcal{O}_0$ is empty
\vspace{.25cm}
\vspace{.25cm}
\item[H4.] $L_1L_2<0$, where 
\begin{equation*}
L_1=g(+\infty)\cdot\sum_{\mathcal{O}_{++}\cup \mathcal{O}_{-+}}[\Psi(i)]_n+g(-\infty)\cdot\sum_{\mathcal{O}_{+-}\cup \mathcal{O}_{--}}[\Psi(i)]_n
\end{equation*}
and
\begin{equation*}
L_2=g(-\infty)\cdot\sum_{\mathcal{O}_{++}\cup \mathcal{O}_{-+}}[\Psi(i)]_n+g(+\infty)\cdot\sum_{\mathcal{O}_{+-}\cup \mathcal{O}_{--}}[\Psi(i)]_n
\end{equation*}
\vspace{.25cm}
 \hspace{-1.75cm}Then there exists a solution to the nonlinear boundary value problem \eqref{E:Scalar}-\eqref{E:SBoundary}.
\end{enumerate}
 \end{theorem}
 Theorem \ref{T:App} is a simple consequence of Theorem \ref{T:Main}.  To see this, suppose the conditions of Theorem \ref{T:App} hold and assume $L_2<0<L_1$. We will abuse notation, slightly, and use $g$ to denote the time dependent function defined on $\{0,\cdots, N\}\times \R$ by $g(t,x)=g(x)$.  Since $g(\pm\infty)$ exist, we must have that $g$ is bounded. Let $\ep>0$ and define the functions $W_1, U_1,  W_2, U_2, w_1 , u_1, w_2$ and $u_2$ in Theorem \ref{T:Main} as follows: 
$W_1(t)=g(+\infty)-\ep$, $U_1(t)=g(-\infty)+\ep$, $W_2(t)=g(+\infty)+\ep$, $U_2(t)=g(-\infty)-\ep$, $w_1(t)=g(-\infty)-\ep$, $u_1(t)=g(+\infty)+\ep$, $w_2(t)=g(-\infty)+\ep$, $u_2(t)=g(+\infty)-\ep$. It is clear that for these functions there exists an $R$, depending on $\ep$, such that \eqref{E:LL}  of Remark \ref{R:LL} holds. 

  Now,  if we calculate $\ds{J_1=\sum_{i=0}^{N-1}[\Psi(i)]_nK_1(i)}$, we get
\begin{equation*}
\begin{split}
\sum_{\mathcal{O}_{++}}[\Psi(i)]_n(g(+\infty)-\ep)&+\sum_{\mathcal{O}_{+-}}[\Psi(i)]_n(g(-\infty)-\ep)\\
&+\sum_{\mathcal{O}_{-+}}[\Psi(t)]_n(g(+\infty)+\ep)+\sum_{\mathcal{O}_{--}}[\Psi(i)]_n(g(-\infty)+\ep),
\end{split}
\end{equation*}
or 
\begin{equation*}
\begin{split}
g(+\infty)\cdot\sum_{\mathcal{O}_{++}\cup \mathcal{O}_{-+}}[\Psi(i)]_n+g(-\infty)\cdot\sum_{\mathcal{O}_{+-}\cup \mathcal{O}_{-+}}[\Psi(t)]_n-\sum_\mathcal{O}|[\Psi(i)]_n|\ep dt.
\end{split}
\end{equation*}
However, this is equal to $\ds{L_1-\sum_\mathcal{O}|[\Psi(i)]_n|\ep .}$ Similarly, $\ds{J_2=L_2+\sum_\mathcal{O}|[\Psi(i)]_n|\ep }$.  Since we are assuming $L_2<0<L_1$,  it is easy to see that for small enough $\ep$, $J_2<0<J_1$. The case where  $L_1<0<L_2$ follows from Remark \ref{R:Reverse} by a similar argument. The result is now a consequence of Remark \ref{R:LL}.

\begin{remark}
The above discussion shows that Theorem \ref{T:Main} is a substantial improvement of the result found in \cite{RodTaylorscalardiscrete}. It is a generalization in two significant ways.  Firstly, Theorem \ref{T:Main} does not require the nonlinearity, $g$, to be bounded as is required in \cite{RodTaylorscalardiscrete} where they impose that $g(\pm\infty)$ exist.  Secondly, the assumptions of Theorem \ref{T:Main} are required only on a bounded interval.  In Theorem \ref{T:App}  the existence of $g(\pm\infty)$ requires very specific behavior of $g$ on intervals of the form $[z_0,\infty)$ and $(-\infty,-z_0]$, for $z_0$ large.
\end{remark}

\subsection{Sturm-Liouville}
In \cite{MaSturm}, the author proves the existence of solutions to nonlinear Sturm-Liouville problems of the form
\begin{equation}\label{E:S}
\Delta(p(t-1)\Delta x(t-1))+q(t)x(t)+\lambda x(t)=f(x(t)); \ts\ts\ts t\in\{a+1, \cdots, b+1\}
\end{equation}
subject to 
\begin{equation}\label{E:SBC}
a_{11}x(a)+a_{12}\Delta x(a)=0 \text{ and } a_{21}x(b+1)+a_{22}\Delta x(b+1)=0,
\end{equation}

where throughout it is assumed that $f:{\R}\to \R, p:[0,1]\to\R$ and $q:[0,1]\to\R$ are continuous, $p(t)>0$ for all $t\in[0,1]$, $a^2+b^2, c^2+d^2>0$, and $\lambda$ is an eigenvalue of the associated linear Sturm-Liouville problem.

Their main result is the following: 
\begin{theorem}\label{T:Sturm}
 Suppose $f:\R\to \R$ satisfies $|f(x)|\leq M_1|x|^\beta +M_2$, where $M_1$ and $M_2$ are nonnegative constants and $\beta\in[0,1)$.  If there exist $z^*$ such that 
\begin{equation*}
\forall z> z^*, f(z)>0  \text{ and }  \forall z<-z^*, f(z)<0,
\end{equation*}
then there exists a solution to \eqref{E:S}-\eqref{E:SBC}.
\end{theorem}

Theorem \ref{T:Sturm} is also a consequence of Theorem \ref{T:Main}.  This follows from the fact that in the case of the Sturm-Liouville problem, because of the self-adjointness associated with it, $[\Psi(i)]_2$ and $[S(i)]_1$ (Theorem \ref{T:Main}), may be chosen to be equal. The result is now a consequence of Corollary \ref{C:Sign} and Remark \ref{R:LL}.
 
 \section{Example}
 We now provide an example which shows the application of Theorem \ref{T:Main}.
Consider
 \begin{equation*}
 y(t+2)+y(t+1)+y(t)=g(y(t+1))\hspace{.25cm} 
  \end{equation*}
 subject to 
 \begin{equation*}
y(5)+y(8)+y(9)=0
 \end{equation*}
 and
 \begin{equation*}
 y(2)+y(8)+y(9)=0
 \end{equation*}
 
   Looking at equations (\ref{E:Scalar}) and (\ref{E:SBoundary}), we see that $n=m=2$.  Writing this in system form, we have 
  \begin{equation*}
    x(t+1)=Ax(t)+f(x(t))
    \end{equation*}
    subject to 
    \begin{equation*}
    B_2x(2)+B_5x(5)+B_8x(8)=0,
    \end{equation*}
  where
  \begin{equation*}
  x(t)=\begin{bmatrix}y(t)\\y(t+1)\end{bmatrix},
  \end{equation*}
 \begin{equation*}
 A=\begin{bmatrix} 0&1\\ -1&-1\\ \end{bmatrix},
 \end{equation*}
 and
 \begin{equation*} B_2=\begin{bmatrix} 0&0\\1&0\end{bmatrix},
B_5=\begin{bmatrix} 1&0\\ 0&0\end{bmatrix} \text{ and }  B_8=\begin{bmatrix} 1&1\\1&1\end{bmatrix}.
\end{equation*}

Since $A$ is constant, it follows that $\Phi(t)=A^t$.  
We then have that 
\begin{equation*}
 B_2\Phi(2)+B_5\Phi(5)+B_8\Phi(8)=\begin{bmatrix}-1&-2\\-1&-2\end{bmatrix}.
 \end{equation*}
 
 If we choose $\begin{bmatrix}2\\-1\end{bmatrix}$ as a basis for $Ker( B_2\Phi(2)+B_5\Phi(5)+B_8\Phi(8))$, then we get 
 \begin{equation*}
 S(t)=
 \begin{cases}
 \begin{bmatrix}2\\-1\end{bmatrix}& \text{ if  } t\equiv 0 \text{ mod } 3\\\\
 \begin{bmatrix}-1\\-1\end{bmatrix}& \text{ if  } t\equiv 1 \text{ mod } 3\\\\
 \begin{bmatrix}-1\\2\end{bmatrix}& \text{ if  } t\equiv 2 \text{ mod } 3
 \end{cases}.
 \end{equation*}
 
We now take $\begin{bmatrix} -1\\1\end{bmatrix}$ as a basis for $Ker\left( (B_2\Phi(2)+B_5\Phi(5)+B_8\Phi(8))^T\right)$, which gives 
\begin{equation*}
\Psi(t)=
\begin{cases}
\begin{bmatrix}0&0\end{bmatrix}^T&\text{ if } t=0,1,5,6,7\\\\
\begin{bmatrix}1&1\end{bmatrix}^T&\text{ if } t=2\\\\
\begin{bmatrix}0&-1\end{bmatrix}^T&\text{ if } t=3\\\\
\begin{bmatrix}-1&0\end{bmatrix}^T&\text{ if } t=4\\\\
\end{cases}.
\end{equation*}
Observe, $\mathcal{O}_{++}=\{2\}, \mathcal{O}_{+-}=\emptyset, \mathcal{O}_{-+}=\emptyset, \mathcal{O}_{--}=\{3\}$, and $\mathcal{O}_0=\emptyset$.

Notice that $[\Psi(t)]_2[S(t)]_2\geq0$ for all $t=0\cdots N-1$, so that Theorem \ref{T:Main} is applicable for an abundance of real-valued functions, $g$, provided $g$ is such that conditions C3*. and C4*. of Corollary \ref{C:Sign} hold for some positive real numbers $c$ and $d$. We point out again, as in Remark \ref{R:LL}, that if C3*. holds eventually, then C4*. is automatically satisfied.

Further, if the end behavior of $g$ is not `sign-changing', then  (\ref{E:Scalar})-(\ref{E:SBoundary}) may still have a solution. It is of interest to note that this may happen for a $g$ which satisfies $\ds{\lim_{x\to\infty}g(\pm x)}=\infty$ (or $\ds{\lim_{x\to\infty}g(\pm x)}=-\infty$), and so is not of standard Landesman-Lazer form.

For a specific instance of this, fix $c>2$ and let $g$ be a continuous function with 
\begin{equation*}
g(x)=
\begin{cases}
\ds{\beta\ln(1+|x|)}& \text{ if } x\in[-d,-c]\\
\ds{\gamma\ln(1+x)}& \text{ if } x\in[c,d]\
\end{cases},
\end{equation*}
where $d=e^{\gamma}(1+c)-1$, $0<\beta<1$, and $\gamma>\ds\frac{\ln(1+c)}{\ln(1+c)-1}$. We will assume that $g(x)<\beta\ln(1+c)$ for $x\in [-c,0]$ and $g(x)<\gamma\ln(1+c)$ for $x\in[0,c]$.

 Let $f(x):=e^x(1+c)-1-(2c+3A\ln(1+c)x^2)$, where $A$ is as in the definition of Theorem \ref{T:Main} and choose $x_c$ such that if $x>x_c$, then $f(x)> 0$. If $\gamma>x_c$, then we have the following:
 \begin{enumerate}
 \item $g(-d)=\beta\ln(1+d)< \ln(1+d)=\gamma+\ln(1+c)<\gamma\ln(1+c)=g(c)$\label{E:sign}
 \item\label{E:C3} \beqq
\begin{split}{cs_{\max}+A\norm{g}_d(s_{\max}+s_{\min})\over s_{\min}}&=2c+3A\gamma\ln(1+d)\\
&<2c+3A\ln(1+c)\gamma^2\\&<e^{\gamma}(1+c)-1\\&=d.
\end{split}
 \eeqq
 \end{enumerate}
 
We now verify that conditions C1.-C4. of Theorem \ref{T:Main} are satisfied for this choice of $g$ and $c,d$.  It has already been noted that $\mathcal{O}_0$ is empty, so that C1. holds.  In this specific example, C2. reduces to the existence of numbers $w_1, w_2, u_1$, and $u_2$ such that 
 \begin{equation*}
 \begin{split}
 &\text{if } x\in[c,d], \text{ then }  w_1<g(x):=g(2,x) \\
 &\text{if } x\in[-d,-c], \text{ then } g(2,x):=g(x)<u_1\\
 &\text{if } x\in[c,d], \text{ then } u_2<g(x)=:g(3,x) \\
 &\text{if } x\in[-d,-c], \text{ then } g(3,x):=g(x)<w_2.
 \end{split}
 \end{equation*}
We take $w_1=g(c), w_2=g(-d), u_1=g(-d)$, and  $u_2=g(c)$, so that C2. is clearly satisfied for these choices of $w_1, w_2, u_1$, and $u_2$. \eqref{E:C3} shows that condition C3.  holds.  Finally, if $J_1$ and $J_2$ are as in Theorem \ref{T:Main}, then we have that $J_1=w_1-w_2=g(c)-g(-d)>0$ and $J_2=u_1-u_2=g(-d)-g(c)<0$ (see \eqref{E:sign}). Thus, C4. is satisfied and we conclude, using Theorem \ref{T:Main}, that there exists a solution to the nonlinear boundary value problem \eqref{E:Scalar}-\eqref{E:SBoundary}.

 \end{document}